\documentclass[english]{smfart}
\usepackage{amssymb}

\newtheorem{theorem}{Theorem}
\newtheorem*{conjecture}{Conjecture}

 \numberwithin{equation}{section}

\newcommand{\be}{\begin{equation}}
\newcommand{\ee}{\end{equation}}
\newcommand{\ba}{\begin{eqnarray}}
\newcommand{\ea}{\end{eqnarray}}

\newcommand{\lab}[1]{\label{#1}}

\newcommand{\C}{\mathbb C}
\newcommand{\R}{\mathbb R}
\newcommand{\Z}{\mathbb Z}

\newcommand{\T}{\mathbb T}

\renewcommand{\L}{\mathbb L}
\newcommand{\m}{\omega}

\newcommand{\eg}{\Gamma_{\! p,q}}

\newtheorem{defn}{Definition}

\begin{document}

\title[Elliptic hypergeometric terms]
{Elliptic hypergeometric terms}

\thanks{Lectures presented at the Workshop ``Th\'eories
galoisiennes et arithm\'etiques des \'equations
diff\'erentielles" (September 2009, CIRM, Luminy, France).}

\author{V. P. Spiridonov}


 \address{Laboratory of Theoretical Physics, JINR, Dubna, Moscow reg.
141980, Russia;
e-mail address: spiridon@theor.jinr.ru}

\begin{abstract}
General structure of the multivariate plain and $q$-hypergeometric terms
and univariate elliptic hypergeometric terms is described.
Some explicit examples of the totally elliptic hypergeometric
terms leading to multidimensional integrals on root systems,
either computable or obeying non-trivial symmetry
transformations, are presented.
\end{abstract}

\maketitle



\tableofcontents

\section{Plain hypergeometric case}

The definition of the hypergeometric series goes
as far back as to Euler and, in a more general setting, to
Pochhammer and Horn \cite{aar,ggr}.

\begin{defn}
The formal series
$$
\sum_{m\in \Z^n} c(m)= \sum_{m\in \Z^n} c(m_1,\dots,m_n)
$$
is called plain hypergeometric series, if the ratios
$$
\frac{c(m_1,\dots,m_i+1,\dots,m_n)}{c(m_1,\dots,m_n)} =R_i(m_1,\dots,m_n)
$$
are rational functions of $m_1,\dots, m_n$.
\end{defn}

Suppose that given rational functions $R_i(m)$, called certificates,
satisfy the consistency conditions
$$
R_i(m_1,\dots ,m_k+1,\dots,m_n )R_k(m)=R_k(m_1,\dots,m_i+1,\dots,m_n )R_i(m).
$$
The general form of corresponding (admissible) plain hypergeometric
 series was determined by Ore and Sato (e.g., see the survey \cite{ggr}).

\begin{theorem}
General admissible plain hypergeometric terms $c(m)$ have the form
$$
c(m)=\frac{R(m)}{\prod_{j=1}^K\Gamma(\epsilon_j(m)+a_j)}\prod_{k=1}^n z_k^{m_k},
$$
where $z_k, a_j$ are arbitrary complex parameters, $K=0,1,2,\dots$,
$\epsilon_j(m)=\sum_{k=1}^n\epsilon_{jk}m_k$, $\epsilon_{jk}\in \Z$,
$R(m)$ is some rational function of $m_1,\dots,m_n$,
and $\Gamma(x)$ is the standard Euler gamma function.
\end{theorem}
Using the inversion formula $\Gamma(x)\Gamma(1-x)=\pi/\sin\pi x$,
some of the $\Gamma$-functions can be put from the denominator of
$c(m)$ to its numerator.

In a similar way one can treat hypergeometric integrals \cite{wz:inv}.
\begin{defn}
The integrals
$$
\int_{D}\Delta(x)dx=\int_{D}\Delta(x_1,\ldots,x_n)dx_1\dots dx_n,
$$
for some domain of integration $D\in\C^n$, are called plain
hypergeometric integrals, if the ratios
$$
\frac{\Delta(x_1\dots,x_i+1,\dots,x_n)}{\Delta(x_1,\dots,x_n)} =R_i(x_1,\dots,x_n)
$$
are rational functions of $x_1,\dots, x_n$.
\end{defn}

The general admissible plain hypergeometric terms $\Delta(x)$ have the form
$$
\Delta(x)=\varphi(x)R(x)
\frac{\prod_{j=1}^K\Gamma(\sum_{k=1}^n\mu_{jk}x_k+b_j)}
{\prod_{j=1}^M\Gamma(\sum_{k=1}^n\epsilon_{jk}x_k+a_j)}\prod_{k=1}^n z_k^{x_k},
$$
where $z_k, a_j,b_j$ are arbitrary complex parameters, $K,M=0,1,2,\ldots$,
$\mu_{jk},\epsilon_{jk}\in\Z$, $R(x)$ is some
rational function of $x_1,\dots,x_n$,
and $\varphi(x)=\varphi(x_1,\dots,x_i+1,\dots,x_n)$ is an
arbitrary periodic function. The plain hypergeometric series can
be obtained from integrals as sums of residues for particular sequences
of poles of $\Delta(x)$.

The $\Gamma(x)$-function can be defined as a special
meromorphic solution of the functional equation
$$
\Gamma(x+1)=x\Gamma(x),
$$
the general solution of which has the form $\varphi(x)\Gamma(x)$,
where $\varphi(x+1)=\varphi(x)$ is an arbitrary periodic function.

For $n=1$, definition 2 yields the Meijer function for the choices $\varphi=1$,
$R=1$ and $D$ one of the contours
1) \{$-{\textup i}\infty,+{\textup i}\infty\}$ separating
sequences of equidistant poles going to the left and to the right
of the complex plane,
2) $\{-\infty-{\textup i}A,-\infty+{\textup i}B\}$ encircling
sequences of poles going to the left
(for some choice of the positive constants $A$ and $B$),
3) $\{+\infty-{\textup i}A,+\infty+{\textup i}B\}$ encircling
sequences of poles going to the right
(for some choice of the constants $A$ and $B$).

It is worth to remark that a limiting form of the plain hypergeometric
terms is determined by the system of partial differential equations
$$
\frac{1}{\Delta(x)}\frac{\partial\Delta(x)}{\partial x_i}=R_i(x).
$$

\section{$q$-hypergeometric case}

$q$-deformations of hypergeometric functions were introduced by
Heine a long time ago \cite{aar}.

\begin{defn}
The formal series
$$
\sum_{m\in \Z^n} c(m)= \sum_{m\in \Z^n} c(m_1,\dots,m_n)
$$
is called $q$-hypergeometric, if
$$
\frac{c(m_1,\dots,m_i+1,\dots,m_n)}{c(m_1,\dots,m_n)} =R_i(q^{m_1},\dots,q^{m_n})
$$
are rational functions of $q^{m_1},\dots, q^{m_r}$, where $q$ is an arbitrary
complex parameter.
\end{defn}

This is a natural extension of the previous definition and it
leads \cite{ggr} to the following theorem.

\begin{theorem}
General admissible $q$-hypergeometric terms $c(m)$ have the form
$$
c(m)=R(q^m)\frac{\prod_{j=1}^K(a_j;q)_{\mu_j(m)}}
{\prod_{j=1}^M(b_j;q)_{\epsilon_j(m)}}\prod_{k=1}^n x_k^{m_k},
$$
where $x_k, a_j, b_j$ are arbitrary complex parameters, $K,M=0,1,2,\ldots,$
$\mu_j(m)=\sum_{k=1}^n\mu_{jk}m_k$ and
$\epsilon_j(m)=\sum_{k=1}^n\epsilon_{jk}m_k$
with $\mu_{jk},\epsilon_{jk}\in\Z$, $R(q^m)$ is some rational function, and
$$
(x;q)_n:=\begin{cases}
 \prod_{j=0}^{n-1}(1-xq^j), & \text{for $n>0$} \\
\prod_{j=1}^{-n} (1-xq^{-j})^{-1}, & \text{for $n<0$}
\end{cases}
$$
is the $q$-shifted factorial (or the $q$-Pochhammer symbol).
\end{theorem}
\begin{defn}

The integrals
$$
\int_{D}\Delta(x)dx=\int_{D}\Delta(x_1,\ldots,x_n)dx_1\dots dx_n,
$$
for some domain of integration $D\in\C^n$, are called
$q$-hypergeometric, if the ratios
$$
\frac{\Delta(x_1,\dots,x_i+1,\dots,x_n)}{\Delta(x_1,\dots,x_n)}
=R_i(q^{x_1},\dots,q^{x_n})
$$
are rational functions of $q^{x_1},\dots, q^{x_n}$.
\end{defn}

Define $q$-gamma functions as special meromorphic
solutions of the finite-difference equation
\be
f(u+\omega_1)=(1-e^{2\pi {\textup i} u/\omega_2})f(u),
\lab{q-eqn}\ee
where $\omega_1,\omega_2\in\C$. Evidently, solutions of this equation
are defined modulo multiplication by an arbitrary $\varphi(u+\omega_1)=\varphi(u)$
periodic function. Introducing the variables
$$
q:=e^{2\pi {\textup i}\omega_1/\omega_2}, \qquad    z:=e^{2\pi {\textup i}u/\omega_2},
$$
this equation can be replaced by
$$
\Gamma_q(qz)=(1-z)\Gamma_q(z).
$$
For $|q|<1$ its particular solution,
analytic near the point $z=0$, is determined by a simple iteration,
which yields
$$
\Gamma_q(z)=\frac{1}{(z;q)_\infty}=\prod_{j=0}^\infty \frac{1}{1-zq^j},
 \quad \Gamma_q(0)=1.
$$
This function can be considered as a $q$-gamma function for $|q|<1$.
More precisely, the  Thomae-Jackson $q$-gamma function has the form
$$
\Gamma(u;q)=(1-q)^{1-u}\frac{(q;q)_\infty}{(q^u;q)_\infty},\qquad
\lim_{q\to1} \Gamma(u;q)=\Gamma(u).
$$
It satisfies the equations
$$
\Gamma(u+1;q)=\frac {1-q^u}{1-q}\Gamma(u;q),\qquad
\Gamma(u-\frac{2\pi {\textup i}}{\log q};q )=(1-q)^{\frac{2\pi {\textup i}}{\log q }} \Gamma(u;q).
$$
For $|q|<1$ the $q$-Pochhammer symbol can be written as
$$
(t;q)_n= \frac{(t;q)_\infty}{(tq^n;q)_\infty},\qquad n\in\Z.
$$

For $|q|=1$ the equation $f(qz)=(1-z)f(z)$ does not have meromorphic
solutions for $z\in\C^*$. In this case it is necessary to consider
equation (\ref{q-eqn}) and search for its solutions meromorphic in $u\in\C$.
The modified $q$-gamma function
\be
\gamma(u;\omega_1,\omega_2)=
\exp\left(-\int_{\R+{\textup i}0}\frac{e^{ux}}
{(1-e^{\omega_1 x})(1-e^{\omega_2 x})}\frac{dx}{x}\right),
\label{mod-q-gamma}\ee
where the contour $\R+{\textup i}0$ passes along the real axis turning
over the point $x=0$ from above in an infinitesimal way,
solves  (\ref{q-eqn}) and remains meromorphic for $\omega_1,\omega_2>0$
when $|q|=1$. (\ref{mod-q-gamma}) is known as the Barnes-Shintani
``double sign" function, the noncompact quantum dilogarithm,
or the hyperbolic gamma function
(see the survey \cite{spi:sur} for relevant references).

We assume that $\text{Re}(\omega_1), \text{Re}(\omega_2)>0$. Then the integral
\eqref{mod-q-gamma} is convergent for
$0<\text{Re}(u)< \text{Re}(\omega_1+\omega_2)$.
Under appropriate restrictions  on $u$ and $\omega_{1,2}$,
it can be computed as a convergent sum
of residues of the poles in the upper half-plane. For
$\text{Im}(\omega_1/\omega_2)>0$ (i.e., for $|q|<1$) this leads to the expression
\begin{equation}
\gamma(u;\omega_1,\omega_2)=\frac{(e^{2\pi {\textup i}u/\omega_1}\tilde q;\tilde q)_\infty}
{(e^{2\pi {\textup i} u/\omega_2}; q)_\infty},
\label{mod-gam-prod}\end{equation}
which is continued analytically to the whole complex plane of $u$.
Here
$$
\tilde q:=e^{-2\pi {\textup i}\omega_2/\omega_1}
$$
is a particular modular transform of $q$.

The function $\gamma(u;\omega_1,\omega_2)$ is symmetric in $\omega_1,\omega_2$.
Therefore it satisfies the second equation
\be
f(u+\omega_2)=(1-e^{2\pi {\textup i} u/\omega_1})f(u).
\lab{q-eqn2}\ee
Suppose $\omega_1,\omega_2$ are real and linearly independent over $\Z$,
i.e. $|q|=1$, but $q^n \neq 1$.
Then, equations (\ref{q-eqn}) and  (\ref{q-eqn2}) taken together
define their solution $\gamma(u;\omega_1,\omega_2)$
uniquely up to multiplication by a constant. For $|q|<1$, there
is a functional freedom in multiplication of $\gamma(u;\omega_1,\omega_2)$
by an arbitrary elliptic function
$$
\varphi(u+\omega_1)=\varphi(u+\omega_2)=\varphi(u).
$$

For $|q|<1$, the general admissible $q$-hypergeometric
term $\Delta(x)$ (integral kernel) has the form
$$
\Delta(x)=\varphi(x)R(q^x)
\frac
{\prod_{j=1}^M(w_jq^{\sum_{k=1}^n\epsilon_{jk}x_k};q)_\infty}
{\prod_{j=1}^K(t_jq^{\sum_{k=1}^n\mu_{jk}x_k};q)_\infty}\prod_{k=1}^n z_k^{x_k},
$$
where $z_k, t_j,w_j$ are arbitrary complex parameters, $K,M=0,1,2,\ldots$,
$\mu_{jk},\epsilon_{jk}\in\Z$, $R(q^x)$ is an arbitrary
rational function of $q^{x_1},\dots,q^{x_n}$,
and $\varphi(x)=\varphi(x_1,\dots,x_i+1,\dots,x_n)$ is
an arbitrary periodic function.

If we drop the factor $\varphi(x)$, then the function
$\rho(q^{x}):=\Delta(x)$ satisfies the system
of $q$-difference equations
$$
 \rho(\dots qy_i\dots) =R_i(y)\rho(y_1,\dots,y_n),
$$
where $R_i(y)$ are rational functions such that
$$
R_i(y_1,\dots, qy_k,\dots, y_n )R_k(y)
=R_k(y_1,\dots, qy_i,\dots, y_n )R_i(y).
$$
Replacement of linear differences by $q$-shifts essentially simplifies formulae.

To construct $\Delta(x)$ for $|q|=1$ one needs the modified $q$-gamma function
$\gamma(x;\omega_1,\omega_2)$. It is natural to require symmetry of
$\Delta(x)$ in $\omega_{1}$ and $\omega_{2}$, i.e., to require
fulfillment of the system of equations
$$
\Delta(\dots,x_i+\omega_k,\dots )
=R_i(e^{2\pi {\textup i}x_1/\omega_{k+1}},\dots,e^{2\pi {\textup i}x_n/\omega_{k+1}};
\omega_k,\omega_{k+1}) \Delta(x),
$$
where $i=1,\ldots, n$ and $k=1,2$, $\ \omega_{k+2}=\omega_k$.
This yields uniquely up to multiplication by a constant
$$
\Delta(x)= \exp{\Bigl(\frac{\pi {\textup i}}{\omega_1\omega_2}
 \sum_{j,k=1}^n\mu_{jk}x_jx_k +\sum_{j=1}^nc_jx_j\Bigr)}\
\prod_{j=1}^M\gamma\Bigl(\sum_{k=1}^n\epsilon_{jk}x_k+a_j;
\omega_1,\omega_2\Bigr),
$$
where $\epsilon_{jk}, \mu_{jk}\in\Z$ and $a_j, c_j\in\C$.

\section{Elliptic case}

The key analytic object of the theory of elliptic functions is the
theta series having a convenient multiplicative form due to the
Jacobi triple product identity
$$
\sum_{n\in \Z} p^{n(n-1)/2}(-x)^n=(p;p)_\infty \theta_p(x),
$$
where
$$
\theta_p(x)=(x;p)_\infty (px^{-1};p)_\infty, \quad
(x;p)_\infty=\prod_{j=0}^\infty (1-xp^j)
$$
for any $x\in\C^*$ and  $p\in\C, |p|<1$.
This function obeys the symmetry properties
$$
\theta_p(x^{-1})=\theta_p(px)=-x^{-1}\theta_p(x)
$$
and the addition law
$$
\theta_p(xw^{\pm 1},yz^{\pm 1}) -\theta_p(xz^{\pm 1},yw^{\pm 1})
=yw^{-1}\theta_p(xy^{\pm 1},wz^{\pm 1}),
$$
where $x,y,w,z\in\C^*$.
We use the conventions
$$
\theta_p(x_1,\ldots,x_k)=\prod_{j=1}^k\theta_p(x_j), \qquad
\theta_p(tx^{\pm1})=\theta_p(tx,tx^{-1}).
$$

For arbitrary $q\in \C$ and $n\in\Z$,  the
elliptic shifted factorials are defined as
$$
\theta_p(x;q)_n:=\begin{cases}
 \prod_{j=0}^{n-1}\theta_p(xq^j), & \text{for $n>0$} \\
\prod_{j=1}^{-n} \theta_p(xq^{-j})^{-1}, & \text{for $n<0$}
\end{cases}
$$
and $\theta_p(x;q)_0=1$.
For $p=0$ we have $\theta_0(x)=1-x$ and
$\theta_0(x;q)_n = (x;q)_n$.

For arbitrary $m\in\Z$, we have the quasiperiodicity relations
$$
\theta_p(p^mx)=(-x)^{-m}p^{-\frac{m(m-1)}{2}}\theta_p(x), \quad
$$
$$
\theta_p(p^mx;q)_k=(-x)^{-mk}q^{-\frac{mk(k-1)}{2}}
p^{-\frac{km(m-1)}{2}}\theta_p(x;q)_k, \quad
$$
$$
\theta_p(x;pq)_k=(-x)^{-\frac{k(k-1)}{2}}q^{-\frac{k(k-1)(2k-1)}{6}}
p^{-\frac{k(k-1)(k-2)}{6}}\theta_p(x;q)_k.
$$

Elliptic gamma functions are defined as special meromorphic solutions
of the finite difference equation
\be
f(u+\omega_1)=\theta_p(e^{2\pi {\textup i} u/\omega_2})f(u),
\lab{e-gamma-eq}\ee
which passes to \eqref{q-eqn} for $p\to 0$ and fixed $u$.
It is not difficult to see that the double infinite product
\begin{equation}
 \Gamma_{\! p,q}(z)
=\prod_{j,k=0}^\infty\frac{1-z^{-1}p^{j+1}q^{k+1}}{1-zp^{j}q^{k}},
\label{ell-gamma}\end{equation}
where $|p|,|q|<1$ and $z\in\C^*$, satisfies the equations
\begin{equation}
\Gamma_{\! p,q}(qz)=\theta_p(z)\Gamma_{\! p,q}(z),\qquad
\Gamma_{\! p,q}(pz)=\theta_q(z)\Gamma_{\! p,q}(z).
\label{e-gam-eq}\end{equation}
The second relation follows from the first one due to the symmetry
in $q$ and $p$. Thus, the function
\begin{equation}
f(u)=\Gamma_{\! p,q}(e^{2\pi {\textup i} u/\omega_2}), \quad q=e^{2\pi {\textup i}\omega_1/\omega_2},
\label{ell-gamma-add}\end{equation}
defines a solution of  equation \eqref{e-gamma-eq} for $|q|<1$.
For fixed $z$, $\Gamma_{\! 0,q}(z)=1/(z;q)_\infty$.

The standard elliptic gamma function (\ref{ell-gamma-add}) is directly related to
the Barnes multiple gamma function of the third order \cite{bar:multiple}.
Its special cases and different properties were investigated by Jackson,
Baxter, Ruijsenaars, Felder, Varchenko, Rains,
and the present author (see survey \cite{spi:sur}).

Suppose three complex variables $\omega_{1,2,3}$ are linearly independent over $\Z$.
Then the well known Jacobi theorem states that if a meromorphic $\varphi(u)$
satisfies the system of equations
$$
\varphi(u+\omega_1)=\varphi(u+\omega_2)=\varphi(u+\omega_3)=\varphi(u),
$$
then $\varphi(u)=const$. Define the bases
$$
q= e^{2\pi {\textup i}\frac{\omega_1}{\omega_2}}, \quad
p=e^{2\pi {\textup i}\frac{\omega_3}{\omega_2}}, \quad  r=e^{2\pi {\textup i}\frac{\omega_3}{\omega_1}}
$$
and their particular modular transformations
$$
\tilde q= e^{-2\pi {\textup i}\frac{\omega_2}{\omega_1}}, \quad
\tilde p=e^{-2\pi {\textup i}\frac{\omega_2}{\omega_3}},   \quad
\tilde r=e^{-2\pi {\textup i}\frac{\omega_1}{\omega_3}}.
$$
The incommensurability condition for $\omega_k$ takes now the
form $p^n\neq q^m$, or $r^n\neq {\tilde q}^m$, etc., $n,m\in\Z$.

The elliptic gamma function (\ref{ell-gamma-add}) can be defined uniquely
as the meromorphic solution of the system of three equations:
\begin{eqnarray*}
f(u+\omega_1)=\theta_p(e^{2\pi {\textup i} u/\omega_2})f(u),
\\
f(u+\omega_2)=f(u),
\\
f(u+\omega_3)=\theta_q(e^{2\pi {\textup i} u/\omega_2})f(u)
\end{eqnarray*}
with the normalization condition $f(\sum_{k=1}^3\omega_k/2)=1$.

The modified elliptic gamma function has the form \cite{spi:theta2}
\be
G(u;\mathbf{\omega})=
\Gamma_{\! p,q}(e^{2\pi {\textup i} \frac{u}{\omega_2}})
\Gamma_{\! \tilde q,r}(re^{-2\pi {\textup i} \frac{u}{\omega_1}}).
\lab{unit-e-gamma}\ee
It defines the unique solution of three equations:
$$
f(u+\omega_2) =\theta_r(e^{2\pi {\textup i} u/\omega_1}) f(u),
\qquad   f(u+\omega_3) =e^{-\pi {\textup i}B_{2,2}(u|\omega_1, \omega_2)} f(u)
$$
and equation \eqref{e-gamma-eq} with the normalization
$f(\sum_{k=1}^3\omega_k/2)=1$. Here
\begin{equation}
B_{2,2}(u|\m_1,\m_2)=\frac{1}{\m_1\m_2}\left(u^2-(\m_1+\m_2)u+\frac{\m_1^2+\m_2^2}{6}+
\frac{\m_1\m_2}{2}\right)
\label{ber2}\end{equation}
is the second Bernoulli polynomial. The function
\be
G(u;\mathbf{\omega})
= e^{-\frac{\pi {\textup i}}{3}B_{3,3}(u|\mathbf{\omega})}
\Gamma_{\! \tilde r,\tilde p}(e^{-2\pi {\textup i} \frac{u}{\omega_3}}),
\lab{mod-e-gamma}\ee
where  $|\tilde p|,|\tilde r|<1$,
\begin{eqnarray*}
&& B_{3,3}(u|\m_1,\m_2,\m_3)=\frac{1}{\m_1\m_2\m_3}\Biggl(u^3-\frac{3u^2}{2}\sum_{k=1}^3\m_k
\\ && \makebox[9em]{}
+\frac{u}{2}\left(\sum_{k=1}^3\m_k^2+3\sum_{j<k}\m_j\m_k\right)
-\frac{1}{4}\left(\sum_{k=1}^3\m_k\right)\sum_{j<k}\m_j\m_k\Biggr)
\end{eqnarray*}
is the third Bernoulli polynomial,
satisfies the same three equations and the normalization
as function \eqref{unit-e-gamma}. Hence they  coincide
and their equality reflects one of the $SL(3;\Z)$
modular group transformation laws \cite{fel-var:elliptic}.
In general it is expected that the elliptic hypergeometric
integrals to be described below represent some
automorphic forms in the cohomology class of $SL(3;\Z)$.

Evidently, one deals in this picture with three elliptic curves
with the modular parameters
$$
\tau_1=\frac{\omega_1}{\omega_2}, \qquad
\tau_2=\frac{\omega_3}{\omega_2}, \qquad
\tau_3=\frac{\omega_3}{\omega_1}, \qquad
$$
satisfying the constraint $\tau_3=\tau_2/\tau_1$.

From expression \eqref{mod-e-gamma} it follows that
$G(u;\mathbf{\omega})$ is a meromorphic function of $u$ for
$\omega_1/\omega_2>0$, when $|q|=1$. Therefore, not surprisingly,
for Im$(\omega_3/\omega_1)$, Im$(\omega_3/\omega_2)\to +\infty$
one has
$$
\lim_{p,r\to0} G(u;\mathbf{\omega}) = \gamma(u;\omega_1,\omega_2).
$$
For $|q|>1$ a solution of (\ref{e-gamma-eq}) is given by
$\Gamma_{\! p,q^{-1}}(q^{-1}e^{2\pi {\textup i} u/\omega_2})^{-1}$.

We use the conventions
\begin{eqnarray*}
&& \Gamma_{\! p,q}(t_1,\ldots,t_k):=\Gamma_{\! p,q}(t_1)\cdots\Gamma_{\! p,q}(t_k),\quad
\\ &&
\Gamma_{\! p,q}(tz^{\pm1}):=\Gamma_{\! p,q}(tz)\Gamma_{\! p,q}(tz^{-1}),\quad
\Gamma_{\! p,q}(z^{\pm2}):=\Gamma_{\! p,q}(z^2)\Gamma_{\! p,q}(z^{-2}).
\end{eqnarray*}
Some useful  properties of $ \Gamma_{\! p,q}(z)$ are:
$$
\theta_p(x;q)_n=\frac{\Gamma_{\! p,q}(xq^n)}{\Gamma_{\! p,q}(x)},
$$
the reflection equation $ \Gamma_{\! p,q}(z)\Gamma_{\! p,q}(pq/z)=1,$
the duplication formula
$$
\Gamma_{\! p,q}(z^2)=\Gamma_{\! p,q}(z,-z,q^{1/2}z,-q^{1/2}z, p^{1/2}z,-p^{1/2}z,
(pq)^{1/2}z,-(pq)^{1/2}z),
$$
and the limiting relation
$$
\lim_{z\to1}(1-z)\Gamma_{\! p,q}(z)=\frac{1}{(p;p)_\infty(q;q)_\infty}.
$$

We skip consideration of elliptic hypergeometric series which are
constructed similarly to the elliptic hypergeometric integrals
which we explain now. Also, we stick to the multiplicative notation,
skipping analysis of the additive finite difference equations for the
integral kernels. Let us define the $n$-dimensional integrals
\[
I(y_1,\ldots, y_m)=\int_{x\in D} \Delta(x_1,\dots,x_n;y_1,\ldots,y_m)
\prod_{j=1}^n\frac{dx_j}{x_j},
\]
where $D\subset \C^n$ is some domain of integration
and $\Delta(x_1,\dots,x_n;y_1,\ldots,
y_m)$ is a meromorphic function
of $x_j, y_k$, where $y_k$ denote the ``external" parameters.
The following definition was introduced in \cite{spi:theta2}.

\begin{defn}
The integral $I(y_1,\ldots, y_m;p,q)$ is called the elliptic hypergeometric
integral if there
are two distinguished complex parameters $p$ and $q$ such that
the kernel $\Delta(x_1,\dots, x_n;y_1,\ldots,y_m;p,q)$
(the elliptic hypergeometric term) satisfies the following
system of linear first order $q$-difference equations in the integration
variables $x_j$:
$$
\frac{\Delta(\ldots qx_j\ldots;y_1,\ldots,y_m;p,q)}
{\Delta(x_1,\dots,x_n;y_1,\ldots,y_m;p,q)}
= h_j(x_1,\dots,x_n;y_1,\ldots,y_m;q;p),
$$
where the $q$-certificates $h_j$, $j=1,\dots, n$, are some $p$-elliptic
functions of the variables $x_k$.
\end{defn}

Let us describe the explicit form of the univariate, $n=1$,
elliptic hypergeometric term $\Delta$. For this we recall that
any meromorphic $p$-elliptic function $f(px)=f(x)$ can be represented
as a ratio of theta functions
$$
f_p(x)=z\prod_{k=1}^N\frac{\theta_p(t_kx)}{\theta_p(w_kx)},\qquad
\prod_{k=1}^Nt_k=\prod_{k=1}^Nw_k,
$$
where $z,t_1,\dots,t_N,w_1,\dots,w_N\in \C^*$ are arbitrary variables
parametrizing the function's divisor. The integer $N=0, 2,3,\dots$ is
called the order of the elliptic function, and, in association with
the hypergeometric functions \cite{spi:sur}, the constraint on products of
the parameters is called the balancing condition.
Since $z=\theta_p(zx,px)/\theta_p(pzx,x)$, the parameter $z$ can be obtained
from the ratios of theta functions by a special choice of
the parameters $t_k$ and $w_k$
without violation of the balancing condition. We can thus set $z=1$.

The equation $\Delta(qx)=f_p(x)\Delta(x)$ determining the elliptic hypergeometric
terms has the general solution for $|q|<1$
$$
\Delta(x)= \varphi(x)\prod_{k=1}^N\frac{\Gamma_{\! p,q}(t_kx)}{\Gamma_{\! p,q}(w_kx)},
\qquad \varphi(x)=\prod_{k=1}^M\frac{\theta_q(a_kx)}{\theta_q(b_kx)},\quad
\prod_{k=1}^Ma_k=\prod_{k=1}^Mb_k,
$$
where $\varphi(qx)=\varphi(x)$ is an arbitrary $q$-elliptic function of order $M$.
Since
$$
\varphi(x)=\prod_{k=1}^M\frac{\Gamma_{\! p,q}(pa_kx,b_kx)}{\Gamma_{\! p,q}(a_kx,pb_kx)},
$$
we see that $\varphi(x)$ can be obtained from the $\Gamma_{\! p,q}$-factors
as a result of a special choice of the parameters $t_k$ and $w_k$ (such
a reduction preserves the balancing condition). So, we can drop $\varphi(x)$
and get the general univariate elliptic hypergeometric term
$$
\Delta(x;t_1,\dots,t_N,w_1,\dots,w_N;p,q) =
\prod_{k=1}^N\frac{\Gamma_{\! p,q}(t_kx)}{\Gamma_{\! p,q}(w_kx)},
\quad \prod_{k=1}^N\frac{t_k}{w_k}=1.
$$
Note that this function is symmetric in $p$ and $q$.
For incommensurate $p$ and $q$, the pair of equations
$$
\Delta(qx)=f_p(x)\Delta(x),\qquad
\Delta(px)=f_q(x)\Delta(x)
$$
determines the kernel $\Delta(x)$ up to a multiplicative constant (which may,
of course, depend on the parameters $t_k$ and $w_k$).

It is easy to see that for $|q|>1$, we have
$$
\Delta(x;t_1,\dots,t_N,w_1,\dots,w_N;p,q) =
\prod_{k=1}^N\frac{\Gamma_{\! p,q^{-1}}(q^{-1}w_kx)}
{\Gamma_{\! p,q^{-1}}(q^{-1}t_kx)}
$$
with the same balancing condition.
The region $|q|=1$ is considered as in the previous section -- it is
necessary to pass to the additive form of the equations for integral kernels
and determine corresponding elliptic hypergeometric terms using the
$G(u;\mathbf{\omega})$-function, which is skipped here.

Although there are some ideas and reasonable arguments how the general
elliptic hypergeometric terms should look like in the multivariable
setting, we prefer to state it as an open problem --- the
formulation of an elliptic (or, more generally, a theta-hypergeometric
\cite{spi:theta2}) analogue of the Ore-Sato theorem.

\section{Totally elliptic hypergeometric terms}

The following definition was introduced by the author in 2001.
\begin{defn}
A meromorphic function $f(x_1,\dots,\, x_n;p)$
of $n+1$ indeterminates $x_j\in\C^*$ and $p\in\C$
with $|p|<1$ or $|p|>1$ is called totally $p$-elliptic if
$$
f(px_1,\ldots,x_n;p)=\ldots=f(x_1,\ldots,px_n;p)=f(x_1,\ldots,x_n;p),
$$
and if its divisor forms a non-trivial
manifold of maximal possible dimension.
\end{defn}

The neat point of this definition consists in the demand of absence of
constraints on the divisor of the elliptic functions except of those
following from the $p$-ellipticity (i.e., positions of zeros and poles
of $f$ should not be prefixed). For instance, the Weierstra\ss{}
function ${\mathcal P}(u)$ with the periods
1 and $\tau$ is not totally elliptic
since the positions of its poles and zeros are fixed. Its linear fractional
transform can be written as a $p$-elliptic function
$f(z)=f(pz)=\theta_p(az,bz)/\theta_p(cz,dz)$,
where $z=e^{2\pi {\textup i} u}$, $p=e^{2\pi {\textup i}\tau}$, $ab=cd$,
but it is not $p$-elliptic for indeterminates $a,b,c$ (if $d$ is counted as a
dependent variable). E.g., the function
$$
f(x,y,w,z;p)=
\frac{\theta_p(xw^{\pm 1},yz^{\pm 1})}{\theta_p(xz^{\pm 1},yw^{\pm 1})}
$$
is totally $p$-elliptic. It was conjectured also that any totally
elliptic function is automatically modular invariant.

Using the notion of total ellipticity and the results
of \cite{spi:theta2,spi:short} it is natural to enrich
the structure of elliptic hypergeometric integrals by
the following definition.

\begin{defn}
An elliptic hypergeometric integral
\[
I(y_1,\ldots, y_m,;p,q)=\int_{x\in D} \Delta(x_1,\dots,x_n;y_1,\ldots,y_m;p,q)
\prod_{j=1}^n\frac{dx_j}{x_j}
\]
is called totally elliptic if all $q$-certificates
$h_j(x_1,\dots,x_n;y_1,\ldots,y_m;q;p)$, including the external
$y_k$-variables certificates
$$
h_{n+k}(x;y;q;p)
=\frac{\Delta(x;\ldots qy_k\ldots;p,q)}
{\Delta(x_1,\dots,x_n;y_1,\ldots,y_m;p,q)}, \quad k=1,\dots, m,
$$
are totally elliptic functions (i.e., they are $p$-elliptic in $x_j, y_k,$ and $q$).
\end{defn}

Some of the properties of such integrals are described in the following theorem.

\begin{theorem}[Rains, Spiridonov, 2004]
Define the meromorphic function
\begin{equation}
\Delta(x_1,\dots,x_n;p,q)
=
\prod_{a=1}^K
\Gamma_{\! p,q}(x_1^{m_1^{(a)}}x_2^{m_2^{(a)}}
\dots x_n^{m_n^{(a)}})^{\epsilon(m^{(a)})},
\label{rs-term}\end{equation}
where $m_j^{(a)}\in\Z$, $j=1,\dots,n,\ a=1,\dots, K$, and
$\epsilon(m^{(a)})=\epsilon(m^{(a)}_1,\dots,m^{(a)}_n)$
are arbitrary  $\Z^n\to \Z$ maps with finite support.
Suppose $\Delta$ is a totally elliptic hypergeometric term, i.e. all its
$q$-certificates are $p$-elliptic functions of $q$ and $x_1,\ldots,x_n$. Then these
certificates are also modular invariant.
\end{theorem}
\begin{proof}
The $q$-certificates have the form
$$
h_i(x;q;p)=\frac{\Delta(\dots qx_i\dots ;p,q)}{\Delta(x_1,\dots,x_n;p,q)}
=\prod_{a=1}^K\theta_p(x^{m^{(a)}};q)_{m_i^{(a)}}^{\epsilon(m^{(a)})}.
$$
The conditions for $h_i$ to be elliptic in $x_j$ are
\begin{align} \nonumber &
1=\frac{h_i(\dots px_j\dots;q;p)}{h_i(x;q;p)}
=\prod_{a=1}^K \Biggl(\Big[-\prod_{l=1}^n x_l^{m_l^{(a)}}\Big]^{-\epsilon(m^{(a)})m_i^{(a)}m_j^{(a)}}
\\ &\makebox[6em]{} \times
q^{-\frac{1}{2}\epsilon(m^{(a)})m_j^{(a)}m_i^{(a)}(m_i^{(a)}-1)}
p^{-\frac{1}{2}\epsilon(m^{(a)})m_i^{(a)}m_j^{(a)}(m_j^{(a)}-1)}\Biggr),
\nonumber\end{align}
which yield the constraints
\begin{eqnarray} \label{con1}
&& \sum_{a=1}^K \epsilon(m^{(a)}) m_i^{(a)} m_j^{(a)} m_k^{(a)} = 0,\\
&& \sum_{a=1}^K \epsilon(m^{(a)}) m_i^{(a)} m_j^{(a)}     = 0
\label{con2}\end{eqnarray}
for $1\le i,j,k\le n$.
The conditions of ellipticity in $q$ have the form
\begin{align} \nonumber &
1=\frac{h_i(x;pq;p)}{h_i(x;q;p)}
=\prod_{a=1}^K \Biggl(\Big[-\prod_{l=1}^n x_l^{m_l^{(a)}}
\Big]^{-\frac{1}{2}\epsilon(m^{(a)})m_i^{(a)}(m_i^{(a)}-1)}
 \\ &\makebox[2em]{} \times
q^{-\frac{1}{6}\epsilon(m^{(a)})m_i^{(a)}(m_i^{(a)}-1)(2m_i^{(a)}-1)}
p^{-\frac{1}{6}\epsilon(m^{(a)})m_i^{(a)}(m_i^{(a)}-1)(m_i^{(a)}-2)}\Biggr).
\nonumber\end{align}
They add the constraint
\begin{equation} \label{con3}
\sum_{a=1}^K \epsilon(m^{(a)}) m_i^{(a)} =0,
\end{equation}
which guarantees that $h_i$ has equal numbers of theta functions
in the numerator and denominator.

Recall the notation
$q=e^{2\pi {\textup i} \omega_1/\omega_2 },\,
\tilde q=e^{-2\pi {\textup i}\omega_2/ \omega_1}$.
In terms of the variable
$\tau=\omega_1/\omega_2$ the full modular group $SL(2,\Z)$ is generated by the
transformations $\tau\to\tau+1$ and $\tau\to\ -1/\tau$. The
function $\theta_q(x)$
is evidently invariant with respect to the first transformation and
$$
\theta_{\tilde q}(e^{-2\pi {\textup i} u/\omega_1})=e^{\pi {\textup i}B_{2,2}(u|\omega_1,\omega_2)}
\theta_q(e^{2\pi {\textup i} u/\omega_2}),
$$
where $B_{2,2}(u|\omega_1,\omega_2)$ is the second Bernoulli
polynomial (\ref{ber2}).

After denoting $x_l=e^{2\pi {\textup i} \gamma_l/\omega_2}$ and
$p=e^{2\pi {\textup i}\omega_3/\omega_2}$,
the conditions of modular invariance of $h_i(x;q;p)$ read
\begin{align} \nonumber &
1=\frac{h_i(e^{-2\pi {\textup i} \gamma/\omega_3};e^{-2\pi {\textup i} \omega_1/\omega_3};
e^{-2\pi {\textup i}\omega_2/\omega_3})}
{h_i(e^{2\pi {\textup i} \gamma/\omega_2};e^{2\pi {\textup i} \omega_1/\omega_2};e^{2\pi {\textup i}\omega_3/\omega_2})}
\\ &\makebox[1em]{}
= \exp\left[\pi {\textup i}\sum_{a=1}^K  \epsilon(m^{(a)})\sum_{l=0}^{m_i^{(a)}-1}
B_{2,2}\Big(\omega_1l+\sum_{j=1}^n\gamma_jm_j^{(a)}\Big|\omega_2,\omega_3\Big)\right],
\nonumber\end{align}
and they are automatically satisfied due to the constraints following
from the total ellipticity.
\end{proof}

Currently the classification of totally elliptic hypergeometric terms and,
in particular, of all solutions $m^{(a)}, \epsilon(m^{(a)})$ of the Diophantine
equations \eqref{con1}, \eqref{con2}, \eqref{con3} is an open problem.
Denote $x=e^{2\pi\textup{i} u/\omega_2}$ and replace $\Gamma_{\! p,q}(x)$-functions
in \eqref{rs-term} by the modified elliptic gamma function $G(u;\mathbf{\omega})$.
Using representation \eqref{mod-e-gamma},
we find
\begin{eqnarray*} && \makebox[-2em]{}
\prod_{a=1}^K G\Big(\sum_{l=1}^n u_l m_l^{(a)};\mathbf{\omega}\Big)^{\epsilon(m^{(a)})}
=e^{\frac{\pi \textup{i}}{12}(\sum_{k=1}^3\omega_k)(\sum_{k=1}^3\omega_k^{-1})
\sum_{a=1}^K\epsilon(m^{(a)})} \Delta(\tilde x_1,\ldots,\tilde x_n;\tilde r,\tilde p),
\end{eqnarray*}
where $\tilde x_j=e^{-2\pi\textup{i} u_j/\omega_3}$.
This simple modular transformation law of the totally elliptic hypergeometric
terms takes place because of the described Diophantine equations
which guarantee that all parameter dependent contributions coming
from $B_{3,3}(u;\mathbf{\omega})$-polynomials cancel out.

The following nontrivial  elliptic hypergeometric term
with $n=6$ and $K=29$ was shown to be totally elliptic in \cite{spi:short}
$$
\Delta(x;t_1,\dots,t_6;p,q)=\frac{\prod_{j=1}^6\Gamma_{\! p,q}(t_jx^{\pm1})}
{\Gamma_{\! p,q}(x^{\pm2})\prod_{1\le i < j\le 6}\Gamma_{\! p,q}(t_it_j)},
\qquad \prod_{j=1}^6t_j=pq,
$$
or, after plugging in $t_6=pq/\prod_{i=1}^5t_i$,
$$
\Delta(x;t_1,\dots,t_5;p,q)=\frac{\prod_{j=1}^5
\Gamma_{\! p,q}(t_jx^{\pm1},t_j^{-1}\prod_{i=1}^5t_i)}
{\Gamma_{\! p,q}(x^{\pm2},\prod_{i=1}^5t_i\,x^{\pm1}))
\prod_{1\le i < j\le 5}\Gamma_{\! p,q}(t_it_j)}.
$$

\begin{theorem}[Spiridonov, 2000] Elliptic beta integral.
If $|t_j|<1$, $j=1,\ldots,6$, then
\begin{equation}
\kappa\int_{\T}\Delta(x;t_1,\dots,t_5;p,q)
\frac{dx}{x}=1,\qquad \kappa=\frac{(p;p)_\infty (q;q)_\infty}{4\pi {\textup i}},
\label{ell-beta}\end{equation}
where $\T$ is the unit circle with positive orientation.
\end{theorem}
\begin{proof} \cite{spi:short}
The partial $q$-difference equation for the kernel
\begin{eqnarray*}
&& \Delta(x;qt_1,t_2,\ldots,t_5;p,q)-\Delta(x;t_1,\ldots,t_5;p,q)
\\ && \makebox[4em]{}
=g(q^{-1}x)\Delta(q^{-1}x;,t_1,\ldots,t_5;p,q)-g(x)\Delta(x;t_1,\ldots,t_5;p,q),
\nonumber\end{eqnarray*}
where
$$
g(x)=
\frac{\prod_{m=1}^5\theta_p(t_mx)}{\prod_{m=2}^5\theta_p(t_1t_m)}
\frac{\theta_p(t_1\prod_{j=1}^5t_j)}
{\theta_p(x^2,x\prod_{j=1}^5t_j)}\frac{t_1}{x},
$$
and its partner obtained after permutation of $p$ and $q$
show that the integral of interest is a constant independent of
the parameters $t_j$.
Giving to $t_j$ special values such that the integral is saturated by
(i.e., its value is given by)
the sum of residues of a fixed pair of poles, we find the needed constant.
\end{proof}

For $p\to 0$ followed by the $t_5\to 0$ and
$$
t_1 = q^{\alpha-1/2}, \ t_2 =-q^{\beta-1/2},
\ t_3=q^{1/2}, \ t_4 = -q^{1/2}, \quad q\to1
$$
limit, equality (\ref{ell-beta}) reduces to the classical Euler beta-integral
evaluation
$$
\int_0^1t^{\alpha-1}(1-t)^{\beta-1}dt=\frac{\Gamma(\alpha)\Gamma(\beta)}
{\Gamma(\alpha+\beta)},\quad
\text{Re}(\alpha),\text{Re}(\beta)>0.
$$
An elliptic analogue of the Gau\ss{} hypergeometric function
${}_2F_1(a,b,;c;z)$ has the form \cite{spi:theta2,spi:sur}
\begin{equation}
V(t_1,\dots,t_8;p,q)=\kappa\int_\T\frac{\prod_{j=1}^8\Gamma_{\! p,q}(t_jx^{\pm 1})}
{\Gamma_{\! p,q}(x^{\pm2})}\frac{dx}{x}, \qquad \prod_{j=1}^8t_j=(pq)^2,
\label{ehf}\end{equation}
where $|t_j|<1$. For $t_7t_8=pq$ (and other similar
restrictions), it reduces to the elliptic beta integral.
This function obeys symmetry transformations attached to the exceptional
root system $E_7$ and satisfies the elliptic hypergeometric equation,
a second order difference equation with some elliptic coefficients
(reducing in a special limit to the standard hypergeometric equation).

\section{Multidimensional integrals}

Let us introduce the meromorphic function
$$
\Delta_n(z,t;p,q)=\prod_{1\leq i<j\leq n} \frac{1}{\eg(z_i^{\pm 1}z_j^{\pm 1})}
\prod_{j=1}^n\frac{\prod_{i=1}^{2n+2m+4}\eg(t_iz_j^{\pm 1})}
{\eg(z_j^{\pm 2})},
$$
and associate with it the following multidimensional elliptic hypergeometric
integral (type I integral for the $BC_n$ root system)
$$
I_n^{(m)}(t_1,\ldots,t_{2n+2m+4})=\kappa_n\int_{\T^n}
\Delta_n(z,t;p,q)\prod_{j=1}^n\frac{dz_j}{z_j},
$$
where $|t_j|<1$,
$$
\prod_{j=1}^{2n+2m+4}t_j=(pq)^{m+1},\qquad
\kappa_n=\frac{(p;p)_\infty^n(q;q)_\infty^n}{2^n n!(2\pi {\textup i})^n}.
$$

\begin{theorem} \cite{die-spi:selberg}
The integral $I^{(m)}_n$ satisfies the $2\binom{2n+2m+4}{n+2}$-set of $(n+2)$-term
recurrences
 \begin{equation}
\sum_{i=1}^{n+2}
\frac{t_i}{\prod_{j=1,\, j\ne i}^{n+2} \theta_p(t_i t_j^{\pm 1})}
I^{(m)}_n(t_1,\dots,qt_i\dots,t_{2n+2m+4})
=0,
 \label{rec-rel1}\end{equation}
where $\prod_{j=1}^{2n+2m+4}t_j=(pq)^mp.$
\end{theorem}

\begin{theorem} \cite{rai-spi:det}
If $t_1\cdots t_{2m+2n+4} = (pq)^{m+1}q$, then we have the following
$2\binom{2n+2m+4}{m+2}$-set of $(m+2)$-term recurrences for $I^{(m)}_n$:
$$
\sum_{1\le k\le m+2}
\frac{\prod_{m+3\le i\le 2n+2m+4} \theta_p(t_it_k/q)}
     {t_k\prod_{1\le i\le m+2;i\ne k} \theta_p(t_i/t_k)}
I^{(m)}_n(t_1,\dots,q^{-1}t_k\dots,t_{2n+2m+4}) =0.
$$
\end{theorem}

Using these two sets of recurrences, for each variable $t_k$
it is possible to define the fundamental matrix function
$M_{ij}\propto I^{(m)}_n($ $\dots, p^{-1}t_i,\dots,q^{-1}t_j,\dots)$,
satisfying the two $\binom{n+m}{n}\times \binom{n+m}{n}$-matrix equations
$$
M(qt_k) = A(t_k)M(t_k), \qquad   M(pt_k) = M(t_k)B(t_k),
$$
for some matrices $A(t_k)$ with $p$-elliptic entries and  $B(t_k)$
with $q$-elliptic entries.
Taken together with the condition of meromorphicity of $M$ in
each variable $t_i$,
this set of equations  determines the matrix $M$
uniquely up to multiplication by a constant \cite{rai-spi:det}.

It can be easily checked that
$$
\prod_{1\leq r<s\leq 2n+2m+4}\eg(t_rt_s)\;
I_m^{(n)}\left(\frac{\sqrt{pq}}{t_1},\ldots,\frac{\sqrt{pq}}{t_{2n+2m+4}}\right)
$$
satisfies the same recurrences as $I_n^{(m)}(t_1,\ldots,t_{2n+2m+4})$
(the recurrences are just swapped after replacing $I^{(m)}_n$
by this expression).
But there is only one solution of the above
system of equations, hence, these two expressions are proportional to each
other and the constant of proportionality is found by residue calculus.

\begin{theorem} \cite{rai:trans}
The integrals $I^{(m)}_n$ satisfy the relation
\begin{eqnarray} \nonumber
&&I_n^{(m)}(t_1,\ldots,t_{2n+2m+4})
\\ && \makebox[2em]{}
=\prod_{1\leq r<s\leq 2n+2m+4}\eg(t_rt_s)\;
I_m^{(n)}\left(\frac{\sqrt{pq}}{t_1},\ldots,\frac{\sqrt{pq}}{t_{2n+2m+4}}\right).
\label{trafo}\end{eqnarray}
\end{theorem}

The main new results of the present paper are described in the following two
theorems. They represent a natural extension of some of the results of
\cite{spi:short} and \cite{rai-spi:det}.

\begin{theorem}
The ratio
$$
\rho(z,y;t;p,q)=\prod_{1\leq r<s\leq 2n+2m+4}\eg(t_rt_s)^{-1}
\frac{\Delta_n(z;t;p,q)}{\Delta_m(y/\sqrt{pq};\sqrt{pq}/t;p,q)}
$$
is a totally elliptic hypergeometric term.
\end{theorem}
\begin{proof}
Explicitly,
\begin{eqnarray*}
&& \rho(z,y;t;p,q)=\prod_{1\leq r<s\leq 2n+2m+4}\frac{1}{\eg(t_rt_s)}
 \frac{\prod_{1\leq i<j\leq m}\eg\left(\frac{y_iy_j}{pq},
\frac{pq}{y_iy_j},\frac{y_i}{y_j},\frac{y_j}{y_i}\right)}
{\prod_{1\leq i<j\leq n}\eg(z_i^{\pm 1}z_j^{\pm 1})}
\\ && \makebox[4em]{} \times
\frac{\prod_{j=1}^m\eg\left(\frac{y_j^2}{pq},\frac{pq}{y_j^2}\right)}
{\prod_{j=1}^n\eg(z_j^{\pm 2})}
\prod_{r=1}^{2n+2m+4}\frac{\prod_{j=1}^n\eg(t_rz_j^{\pm 1})}
{\prod_{j=1}^m\eg\left(\frac{y_j}{t_r},
\frac{pq}{t_ry_j}\right)}.
\end{eqnarray*}
This term contains elliptic gamma functions with non-removable integer
powers of $pq$ in their arguments showing that ansatz (\ref{rs-term})
is not the most general one leading to interesting integrals.
After substitution $y_j\to \sqrt{pq}y_j$, we see that the $\rho$-function becomes
equal to a ratio of the kernels of integrals in (\ref{trafo}).

The $z_k$-variable $q$-certificate is
\begin{eqnarray*}
&& h_{k}^{(z)}(z;t;q;p)=\frac{\rho(\dots qz_k\dots,y;t;p,q)}{\rho(z,y;t;p,q)}
\\ && \makebox[4em]{}
=\frac{\theta_p(q^{-2}z_k^{-2},q^{-1}z_k^{-2})}
{\theta_p(qz_k^{2},z_k^{2})}
\prod_{i=1,\neq k}^n\frac{\theta_p(q^{-1}z_k^{-1}z_i^{\pm1})}
{\theta_p(z_kz_i^{\pm1})}
\prod_{r=1}^{2n+2m+4}\frac{\theta_p(t_rz_k)}
{\theta_p(q^{-1}t_rz_k^{-1})}.
\end{eqnarray*}
The equality $h_{k}^{(z)}(\dots pz_l \dots)=h_{k}^{(z)}(z;t;q;p)$
for $l\neq k$ is automatically true, and for $l=k$ it requires the
balancing condition
$$
\frac{h_{k}^{(z)}(\dots pz_k \dots)}{h_{k}^{(z)}(z;t;q;p)}
=\frac{(pq)^{2m+2}}{\prod_{r=1}^{2n+2m+4}t_r^2}=1.
$$
The equality $h_{k}^{(z)}(\dots pt_i, \dots, p^{-1}t_j\dots)
=h_{k}^{(z)}(z;t;q;p)$ is valid automatically. Most non-trivial
is the ellipticity in $q$ since the balancing condition
contains the multiplier $q^{m+1}$. Assuming that
$t_1=(pq)^{m+1}/\prod_{r=2}^{2n+2m+4}t_r$, we find
\begin{eqnarray*}
&&
\frac{h_{k}^{(z)}(z;p^{m+1}t_1,t_2\dots;pq;p)}{h_{k}^{(z)}(z;t;q;p)}
=\frac{1}{(pqz_k)^{4}}\prod_{i=1,\neq k}^n\frac{1}{(pqz_k)^{2}}
\prod_{r=2}^{2n+2m+4}\frac{-pqz_k}{t_r}\,
\\ && \makebox[13em]{} \times
\frac{(-q^{-1}t_1z_k^{-1})^{m}p^{m(m-1)/2}}
{(-t_1z_k)^{m+1}p^{m(m+1)/2}}
=1.
\end{eqnarray*}
Taken together, these conditions guarantee that $h_k^{(z)}$ is invariant with
respect to all transformations $z_l\to p^{n_l}z_l$, $t_r\to
 p^{m_r}t_r$, $q\to p^Nq$ such that $n_l, m_r, N\in \Z$ and
$\sum_{r=1}^{2n+2m+4}m_r=N(m+1)$.

The $y_k$-variables $q$-certificates are equivalent to $h_{k}^{(z)}$. Therefore
it remains to consider the $t_j$-variables $q$-certificates
\begin{eqnarray*}
&& h_{ik}^{(t)}(z;t;q;p)=\frac{\rho(\dots qt_i,\dots,q^{-1}t_k\dots)}
{\rho(z,y;t;p,q)}
\\ && \makebox[4em]{}
=\prod_{l=1,\neq i,k}^{2n+2m+4}\frac{\theta_p(q^{-1}t_kt_l)}
{\theta_p(t_it_l)}\prod_{j=1}^{n}\frac{\theta_p(t_iz_j^{\pm1})}
{\theta_p(q^{-1}t_kz_j^{\pm1})}
 \prod_{j=1}^{m}\frac{\theta_p(q^{-1}t_i^{-1}y_j,
pt_i^{-1}y_j^{-1})}{\theta_p(pqt_k^{-1}y_j^{-1},t_k^{-1}y_j)}.
\end{eqnarray*}
The ellipticity in $z_j$ and $y_l$ is easy to check.
The invariance of $h_{ik}^{(t)}$
with respect to the transformations $t_a\to pt_a, t_b\to p^{-1}t_b$ for various
choices of the indices $a$ and $b$ is also verified without difficulties.
Finally, assuming that $t_s=(pq)^{m+1}/\prod_{r=1,\neq s}^{2n+2m+4}t_r$, where
$s\neq i,k$, we find
\begin{eqnarray*}
&&
\frac{h_{ik}^{(t)}(\dots p^{m+1}t_s\dots;pq;p)}{h_{ik}^{(t)}(z;t;q;p)}=
\frac{(-t_it_s)^{m+1}p^{m(m+1)/2}}{(-q^{-1}t_kt_s)^mp^{m(m-1)/2}}
\\ && \makebox[8em]{} \times
\prod_{l=1,\neq i,k,s}^{2n+2m+4}\frac{-t_kt_l}{pq}
\prod_{j=1}^n\frac{p^2q^2}{t_k^2}\prod_{j=1}^m\frac{1}{t_it_k} =1,
\end{eqnarray*}
and a similar (though even more complicated) picture holds for $s=i$ or $s=k$,
which completes the proof.
\end{proof}

As mentioned already, sums of residues of particular sequences
of the integral kernel poles form elliptic hypergeometric series.
Under special constraints, exact computation or symmetry transformation
formulae for integrals reduce to particular totally elliptic functions
identities \cite{spi:sur}. An analogue of the total ellipticity requirement for
elliptic beta integrals was (implicitly) found in \cite{spi:short}.
The value of the above theorem consists in the first extension of
the latter property to an integral transformation formula.
In general, this ``ratio of integral kernels" criterion provides a powerful
technical tool for conjecturing new relations between integrals.

Indeed, an analogous result can be established for the following
type I elliptic hypergeometric integral for the $A_n$-root system \cite{spi:theta2}
\begin{eqnarray*}
&& I_{n}^{(m)}(s_1,\ldots,s_{n+m+2};t_1,\ldots,t_{n+m+2})
=\mu_n \int_{\T^n}\Delta_n(z;s,t;p,q) \prod_{j=1}^n\frac{dz_j}{z_j},
\end{eqnarray*}
where
$$
\Delta_n(z;s,t;p,q)=\prod_{1\leq j<k\leq n+1}
\frac{1}{\eg(z_jz_k^{-1},z_j^{-1}z_k)}\prod_{j=1}^{n+1}
\prod_{l=1}^{n+m+2}\eg(s_lz_j,t_lz_j^{-1}),
$$
and $|t_j|, |s_j|<1$,
$$
\prod_{j=1}^{n+1}z_j=1,\qquad \prod_{l=1}^{n+m+2}s_lt_l=(pq)^{m+1},\qquad
\mu_n=\frac{(p;p)_\infty^n(q;q)_\infty^n}{(n+1)!(2\pi {\textup i})^n}.
$$

We denote $T=\prod_{j=1}^{n+m+2}t_j,\, S=\prod_{j=1}^{n+m+2}s_j$,
so that $ST=(pq)^{m+1}$, and let all $|t_k|,\, |s_k|,\, |T^{\frac{1}{m+1}}/t_k|,\,
|S^{\frac{1}{m+1}}/s_k|<1$. Then \cite{rai:trans}
\begin{eqnarray}\nonumber
&&
I_{n}^{(m)}(s_1,\ldots,s_{n+m+2};t_1,\ldots,t_{n+m+2})
=\prod_{j,k=1}^{n+m+2}\eg(t_js_k) \\ && \makebox[1em]{} \times
I_{m}^{(n)}\left(\frac{S^{\frac{1}{m+1}}}{s_1},\ldots,
\frac{S^{\frac{1}{m+1}}}{s_{n+m+2}};
\frac{T^{\frac{1}{m+1}}}{t_1},\ldots,\frac{T^{\frac{1}{m+1}}}{t_{n+m+2}}
\right).
\label{A_n}
\end{eqnarray}

\begin{theorem}
The ratio
$$
\rho(z,y;s,t;p,q)=\frac{\Delta_n(z;s,t;p,q)}{\Delta_m(y;s^{-1},pqt^{-1};p,q)}
\prod_{k,r=1}^{n+m+2}\frac{1}{\eg(s_kt_r)},
$$
where $\prod_{j=1}^nz_j=1,\ \prod_{j=1}^my_j= S,$
is a totally elliptic hypergeometric term.
\end{theorem}
\begin{proof}
Explicitly,
\begin{eqnarray*}
&& \rho(z,y;s,t;p,q)=\prod_{1\leq k,r\leq n+m+2}\frac{1}{\eg(s_kt_r)}
 \prod_{r=1}^{n+m+2}\frac{\prod_{j=1}^{n+1}\eg(s_rz_j,t_rz_j^{-1})}
{\prod_{j=1}^{m+1}\eg(s_r^{-1}y_j,pqt_r^{-1}y_j^{-1})}
\\ && \makebox[14em]{} \times
\frac{\prod_{1\leq i<j\leq m+1}\eg\left(y_i^{-1}y_j,y_iy_j^{-1}\right)}
{\prod_{1\leq i<j\leq n+1}\eg(z_i^{-1}z_j,z_iz_j^{-1})}.
\end{eqnarray*}
After substitution $y_j\to S^{1/(m+1)}y_j$, we see that $\rho$ becomes
equal to the ratio of kernels of the integrals occurring in relation (\ref{A_n}).

The $z_k$-variables $q$-certificates have the form
\begin{eqnarray*} && \makebox[-2em]{}
 h_{k}^{(z)}(z;s,t;q;p)=\frac{\rho(\dots qz_k\dots,q^{-1}z_{n+1},y;t;p,q)}
{\rho(z,y;t;p,q)} =\frac{\theta_p(q^{-2}z_k^{-1}z_{n+1},q^{-1}z_k^{-1}z_{n+1})}
{\theta_p(qz_kz_{n+1}^{-1},z_kz_{n+1}^{-1})}
\\ && \makebox[4em]{} \times
\prod_{i=1,\neq k}^n\frac{\theta_p(q^{-1}z_iz_k^{-1},q^{-1}z_i^{-1}z_{n+1})}
{\theta_p(z_i^{-1}z_k,z_i,z_{n+1}^{-1})}
\prod_{r=1}^{n+m+2}\frac{\theta_p(s_rz_k,t_rz_{n+1}^{-1})}
{\theta_p(q^{-1}s_rz_{n+1},q^{-1}t_rz_k^{-1})}.
\end{eqnarray*}
Direct computations yield
\begin{eqnarray}
&&
\frac{h_{k}^{(z)}(\dots pz_l, \dots, p^{-1}z_{n+1} \dots)}
{h_{k}^{(z)}(z;s,t;q;p)}
= \frac{(pq)^{m+1}}{\prod_{r=1}^{n+m+2}s_rt_r}
=1, \quad l\neq k,
\nonumber \\ &&
\frac{h_{k}^{(z)}(\dots pz_k, \dots, p^{-1}z_{n+1} \dots)}
{h_{k}^{(z)}(z;s,t;q;p)}
= \frac{(pq)^{2(m+1)}}{\prod_{r=1}^{n+m+2}s_r^2t_r^2}=1, \nonumber \\
&&
\frac{h_{k}^{(z)}(\dots ps_a, \dots, p^{-1}s_b \dots)}
{h_{k}^{(z)}(z;s,t;q;p)}= \frac{h_{k}^{(z)}(\dots ps_a, \dots, p^{-1}t_b \dots)}
{h_{k}^{(z)}(z;s,t;q;p)}=1.
\end{eqnarray}
The most complicated is the verification of the equality
$$
\frac{h_{k}^{(z)}(\dots p^{m+1}t_a \dots ;pq;p)}
{h_{k}^{(z)}(z;s,t;q;p)}=\frac{(pq)^{m+1}}{\prod_{r=1}^{n+m+2}s_rt_r}=1.
$$

The $y_k$-variables $q$-certificates are equivalent to $h_{k}^{(z)}$. Therefore
we can pass to the $t_j,s_k$-variables $q$-certificates (reduced by their
permutational symmetry)
\begin{eqnarray*}
&& h_{il}^{(t)}(z,y;s,t;q;p)=\frac{\rho(\dots qt_i,\dots,q^{-1}t_l\dots)}
{\rho(z,y;s,t;p,q)}
\\ && \makebox[4em]{}
=\prod_{k=1}^{n+m+2}\frac{\theta_p(q^{-1}s_kt_l)}
{\theta_p(s_kt_i)}\prod_{j=1}^{n+1}\frac{\theta_p(t_iz_j^{-1})}
{\theta_p(q^{-1}t_lz_j^{-1})}
 \prod_{j=1}^{m+1}\frac{\theta_p(pt_i^{-1}y_j^{-1})}
{\theta_p(pqt_l^{-1}y_j^{-1})}.
\end{eqnarray*}
This function is automatically invariant with respect to the
replacements $t_a\to pt_a$ for any $a$ and fixed other variables,
$q\to pq$ (without touching other parameters)
as well as $s_a\to ps_a,\ s_b\to p^{-1}s_b$.

The mixing $q$-certificates have the form
\begin{eqnarray*}
&& h_{ijl}^{(y,s,t)}(z,y;s,t;q;p)=\frac{\rho(\dots qy_i,\dots,qs_j,\dots
q^{-1}t_l,\dots)}
{\rho(z,y;s,t;p,q)}
\\ && \makebox[4em]{}
=\frac{\prod_{r=1,\neq j}^{n+m+2}\theta_p(q^{-1}s_rt_l)}
{\prod_{r=1,\neq l}^{n+m+2}\theta_p(s_jt_r)}
\prod_{k=1}^{n+1}\frac{\theta_p(s_jz_k)}
{\theta_p(q^{-1}t_lz_k^{-1})}
 \prod_{k=1,\neq i}^{m+1}\frac{\theta_p(q^{-1}s_j^{-1}y_k)}
{\theta_p(pqt_l^{-1}y_k^{-1})}
\\ && \makebox[6em]{} \times
\frac{\prod_{r=1,\neq l}^{n+m+2}\theta_p(pt_r^{-1}y_i^{-1})}
{\prod_{r=1,\neq j}^{n+m+2}\theta_p(s_r^{-1}y_i)}
\prod_{k=1, \neq i}^{m+1}\frac{\theta_p(y_iy_k^{-1})}
{\theta_p(q^{-1}y_i^{-1}y_k)}.
\end{eqnarray*}
The invariance of this function with respect to the transformations
described for other certificates is verified in a direct way.
E.g., the most complicated case is
\begin{eqnarray*}
&& \frac{h_{ijl}^{(y,s,t)}(\dots p^{m+1}t_l \dots ;pq;p)}
{h_{ijl}^{(y,s,t)}(z,y;s,t;q;p)}
=\frac{\prod_{k=1}^{n+1}(-q^{-1}t_lz_k^{-1})^mp^{m(m-1)/2}}
{\prod_{r=1,\neq j}^{n+m+2}(-q^{-1}s_rt_l)^mp^{m(m-1)/2}}
\\ && \makebox[4em]{}  \times
\prod_{k=1,\neq i}^{m+1}\frac{(-y_k)^{m+1}p^{m(m+1)/2}}
{(pq)^{m+1}s_j t_l^{-m}} \prod_{k=1,\neq i}^{m+1}\frac{-pqy_i}{y_k}=1.
\end{eqnarray*}
To summarize, all $q$-certificates have the symmetries
$$
z_k\to p^{n_k}z_k,\quad
y_j\to p^{m_j}y_j,\quad s_l\to p^{\alpha_l}s_l,\quad
t_l\to p^{\beta_l}t_l,\quad q\to p^Nq,
$$
where $n_k,m_j,\alpha_l,\beta_l,N\in\Z$ and $\sum_{k=1}^{n+1}n_k=0$,
$\sum_{j=1}^{m+1}m_j=\sum_{l=1}^{n+m+2}{\alpha_l},$ and
$\sum_{l=1}^{n+m+2}(\alpha_l+\beta_l)=N(m+1).$
This completes the proof of the theorem.
\end{proof}

In the univariate case, $n=1$, both type I $A_n$ and $BC_n$-integrals
coincide with the $V(t_1,\dots,t_8;p,q)$-function described in the
previous section. The corresponding symmetry transformations
are mere consequences of the $E_7$-root system transformation
found in \cite{spi:theta2}. The simplest $p\to 0$ limit of the arbitrary
rank integrals reduces them to the Gustafson integrals \cite{gus:some1}.

In joint work with Vartanov
\cite{SV,spi-var:elliptic} we gave a large list of known simple
transformation formulae for elliptic hypergeometric integrals related to the
Weyl groups (including the cases of computable integrals).
There are about seven proven different multiple elliptic beta integrals
associated to root systems and a similar number of proven symmetry
transformations for higher order elliptic hypergeometric functions.
Additionally, there are about fifteen conjectures for new elliptic beta integrals
and a similar number of new conjectured transformation formulae
for integrals with a bigger number of parameters. In particular,
there are interesting examples of elliptic hypergeometric
terms having fractional powers of $pq$ in the elliptic gamma function
arguments. In such cases the total ellipticity
condition should be modified --- it is necessary to consider dilations by
appropriate powers of $q$ (or $p$).
Using the preceding two theorems as a starting point, for all these
relations we have checked validity of the following conjecture.

\begin{conjecture} \cite{spi-var:elliptic}. The condition of total ellipticity
for elliptic hypergeometric terms is necessary
for the existence of exact evaluation formulae for elliptic beta integrals
or of nontrivial Weyl group symmetry transformations for higher order
elliptic hypergeometric functions (emerging in the ``ratio of integral kernels"
spirit, as described above).
\end{conjecture}

Most of the new relations for integrals were inspired by the
Seiberg dualities for ${\mathcal N}=1$ four dimensional
supersymmetric field theories, where elliptic hypergeometric
integrals play the role of superconformal (topological) indices with an
appropriate matrix integral representation. For further details, see
\cite{spi-var:elliptic}. A few more identities are conjectured in recent
interesting papers by Gadde et al \cite{gprr,grry} as a consequence of
known dualities for ${\mathcal N}=2$ and ${\mathcal N}=4$
four dimensional superconformal
field theories. Here it should be stressed
that there are actually infinitely many symmetry relations for
elliptic hypergeometric integrals generated in a recursive way
from some canonical exact formulae, see \cite{spi:bint}
for a tree of identities following from the univariate elliptic
beta integral and its root system generalizations
following from the integral transforms of \cite{sw}. Therefore
many identities are, in fact, relatively simple consequences of some
universal relations, whose enumeration would be of most interest.

\medskip
{\bf Acknowledgments.} The author is indebted to C. Krattenthaler,
E.~M. Rains, G.~S. Vartanov for valuable discussions and collaboration
and to L.~Di Vizio and T. Rivoal for invitation to this workshop and
kind hospitality. The author is partially supported by the RFBR grants
08-01-00392 and 09-01-93107-NCNIL-a, and the Max Planck Institute for
Mathematics (Bonn, Germany) where this paper was completed.

\end{document}